\documentclass[12pt]{amsart}
\usepackage{amssymb}
\usepackage{amsthm}
\usepackage{amsmath}
\usepackage{graphicx}
\usepackage{mathrsfs}
\usepackage{caption}
\usepackage[usenames,dvipsnames]{xcolor}
\usepackage{bm}
\usepackage{xspace}
\usepackage{enumitem}
\setlist[description]{leftmargin=\parindent, labelindent=\parindent}
\setlist[itemize]{nosep}
\usepackage[normalem]{ulem}
\usepackage{outlines}
\usepackage{hyperref}
\usepackage{geometry}
\usepackage{marginnote}
\usepackage{xparse}
\usepackage{mathtools}

\DeclareMathAlphabet{\mathpzc}{OT1}{pzc}{m}{it}
\usepackage[only,llbracket,rrbracket]{stmaryrd}
\SetSymbolFont{stmry}{bold}{U}{stmry}{m}{n}

\newcommand{\bF}{\mathbb{F}}

\newcommand{\bz}{\mathbb{Z}}

\newcommand{\br}{\mathbb{R}}

\newcommand{\al}{\alpha}
\newcommand{\be}{\beta}
\newcommand{\ga}{\gamma}

\renewcommand{\Im}{\mathrm{Im}}

\newcommand{\Ozsvath}{Ozsv{\'a}th\xspace}
\newcommand{\Szabo}{Szab{\'o}\xspace}

\DeclareMathOperator{\Id}{Id}
\def\co{\colon\thinspace}

\DeclareMathOperator{\CKh}{CKh}
\DeclareMathOperator{\Kh}{Kh}

\DeclareMathOperator{\Sz}{Sz}
\DeclareMathOperator{\CSz}{CSz}
\newcommand{\rCSz}{\widetilde{\CSz}}
\newcommand{\rSz}{\widetilde{\Sz}}

\newcommand{\HF}{\widehat{\mathrm{HF}}}
\DeclareMathOperator{\diagram}{\mathcal{D}}

\newcommand{\handle}[1][]{\mathfrak{h}_{ #1 }}
\DeclareMathOperator{\cone}{cone}
\DeclareMathOperator{\cob}{\mathpzc{Cob}}
\DeclareMathOperator{\matcob}{\mathpzc{Mat}(\cob)}

\newcommand{\skft}{\mathcal{K}}
\newcommand{\BN}[1]{\llbracket #1 \rrbracket}
\newcommand{\rskft}{\widetilde{\skft}}
\newcommand{\kft}{\mathcal{A}}
\newcommand{\rkft}{\tilde{\kft}}

\newcommand{\arrowunder}[1]{\mathrel{\mathop{\longrightarrow}_{#1}}}
\DeclareMathOperator{\Kom}{\mathpzc{Kom}}

\newtheorem{Thm}{Theorem}[section]
\newtheorem*{Thm*}{Theorem}
\newtheorem{Prop}[Thm]{Proposition}
\newtheorem*{Prop*}{Proposition}
\newtheorem{Lem}[Thm]{Lemma}
\newtheorem*{Lem*}{Lemma}
\newtheorem{Cor}[Thm]{Corollary}
\newtheorem*{Conj}{Conjecture}

\theoremstyle{definition}
\newtheorem{Def}[Thm]{Definition}
\newtheorem*{Def*}{Definition}

\theoremstyle{remark}
\newtheorem{Rem}[Thm]{Remark}
\newtheorem{Rem*}{Remark}
\newcounter{Eg}

\newtheorem*{Eg*}{Example}

\newsavebox{\fminipagebox}
\NewDocumentEnvironment{fminipage}{m O{\fboxsep}}
 {\par\kern#2\noindent\begin{lrbox}{\fminipagebox}
  \begin{minipage}{#1}\ignorespaces}
 {\end{minipage}\end{lrbox}%
  \makebox[#1]{%
    \kern\dimexpr-\fboxsep-\fboxrule\relax
    \fbox{\usebox{\fminipagebox}}%
    \kern\dimexpr-\fboxsep-\fboxrule\relax
  }\par\kern#2
 }

\begin{document}
\title[Khovanov-Floer theories and mutation]{Mutation-invariance of Khovanov-Floer theories}
\author{Adam Saltz}
\address{University of Georgia}
\email{adam.saltz@uga.edu}
\date{\today}
\thanks{This material is based upon work supported by the National Science
  Foundation under Grant No. 1664567.}
\begin{abstract}
  Khovanov-Floer theories are a class of homological link invariants which admit
  spectral sequences from Khovanov homology.  They include Khovanov homology,
  \Szabo's geometric link homology, singular instanton homology, and various
  Floer theories applied to branched double covers.  In this short note we show
  that certain \emph{strong} Khovanov-Floer theories, including \Szabo homology
  and singular instanton homology, are invariant under Conway mutation.  Along
  the way we prove two other conjectures about the structure of \Szabo homology.  
\end{abstract}
\maketitle

\section{Introduction}\label{sec:background}

Let $L \subset S^3$ be a link.  Suppose that there is a sphere in $S^3$ which
intersects $L$ transversely in four points, splitting the link into two tangles,
each in a three-ball.  \emph{Mutation} is the operation of regluing the
three-balls by some orientation-preserving involution of the sphere.  Mutation can also be
represented diagrammatically: if $\diagram$ is a link diagram and $s$ is a
circle which intersects $\diagram$ transversely in four points, then mutation is
the operation of rotating the tangle inside $s$ by $180^\circ$.  If $L'$ is the
result of mutation on $L$, then we say that $L$ and $L'$ are \emph{mutants}.

\begin{figure}[h]
  \centering
  \includegraphics[scale=.5]{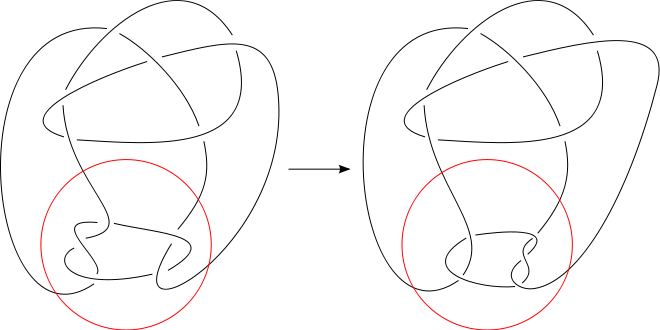} 
  \caption{On the left, the Kinoshita-Terasaka knot (11n42). On the right, the
    Conway knot (11n34). The two are related by mutation along the Conway sphere
    shown in red. \cite{Jorgenson2010} }
\end{figure}

Perhaps surprisingly, many link invariants cannot detect mutation: if $L$
and $L'$ are mutants, then the invariants agree on $L$ and $L'$.  These include
the signature and HOMFLY-PT polynomial (and therefore the Jones and Alexander
polynomials).
Write $\Sigma(L)$ for the double cover of $S^3$ branched along $L$.  Viro
showed that $\Sigma(L)$ and $\Sigma(L')$ are diffeomorphic \cite{Viro1979}.
Therefore any invariant which is also an invariant of branched double covers
cannot detect mutation.  On the
other hand, genus and knot Floer homology can distinguish mutants.  So the ability
to detect mutants is an important measure of the power of a link invariant.

In this note we prove that a certain class of link invariants called \emph{Khovanov-Floer theories} do not detect
mutation.  The definition of a Khovanov-Floer theory was introduced by Baldwin,
Hedden, and Lobb~\cite{BaldwinHeddenLobb2015} 
to encompass the many link homology theories which admit spectral sequences from
Khovanov homology.  These include singular instanton homology, \Szabo's
``geometric'' link homology, and various Floer homologies applied to
$\Sigma(-L)$.  Viro's theorem implies that the latter are all
mutation-invariant, and Khovanov homology (with coefficients in $\bz/2\bz$) is itself is mutation-invariant, see
~\cite{Wehrli2010},~\cite{Bloom2011}, and~\cite{2017arXiv170100880L}.   So it's natural to ask if the other
Khovanov-Floer theories are mutation-invariant as well.  All the theories in this paper have coefficients in $\bF = \bz/2\bz$.
\begin{Thm}\label{thm:main}
  Let $\skft$ be a conic, strong Khovanov-Floer theory over the ring $\bF[X]/(X^2)$.
  Let $\diagram$ and $\diagram'$ be mutants.  Then $H(\skft(\diagram)) \cong
  H(\skft(\diagram'))$.
\end{Thm}

In \cite{Saltz2017a} we defined strong Khovanov-Floer theories and showed that \Szabo
homology and singular instanton homology count among them. 

\begin{Cor}\label{cor:main}
  \Szabo homology and singular instanton homology are mutation-invariant.
\end{Cor}

The proof of the theorem uses only the structure of strong Khovanov-Floer theories and makes no
reference to, say, instantons.  The corollary was conjectured by Peter
Lambert-Cole in~\cite{2017arXiv170100880L} (and earlier by Seed, see below).  He defined a particular class of
Khovanov-Floer theories, so-called \emph{extended} theories, characterized by
the fact that their reduced versions satisfy a K\"{u}nneth theorem for connected sums.

\begin{Thm*}[Lambert-Cole, \cite{2017arXiv170100880L}]
  Extended Khovanov-Floer theories are mutation-invariant.
\end{Thm*}

Lambert-Cole conjectures that \Szabo homology and singular instanton homology
are extended Khovanov-Floer theories.  We prove Theorem \ref{thm:main} by proving that every conic, strong
Khovanov-Floer theory over $\bF[X]/(X^2)$ is extended.  The advantage of working with conic, strong Khovanov-Floer theories is that they
``factor through'' Bar-Natan's formal Khovanov bracket; this is the main
technical result of \cite{Saltz2017a}. (See Proposition
\ref{prop:mainTech} in the next section for a precise statement.)  So nearly any
property of Khovanov homology which can be proved via the formal bracket is also a
property of conic, strong Khovanov-Floer theories.  Therefore it suffices to frame the K\"{u}nneth
formula for reduced Khovanov homology in terms of the formal bracket.

There are two other extant proofs that Khovanov homology is mutation-invariant.
Wehrli's was the first~\cite{Wehrli2010}, and it applies only to
component-preserving mutations.  It uses some subtle structure of the formal bracket; in
fact, it uses more structure than Proposition \ref{prop:mainTech} provides.
Unfortunately, we are not yet able to show that strong Khovanov-Floer theories
carry all this structure.  If they do, then one could upgrade Theorem
\ref{thm:main} to a statement about chain homotopy types rather than homology
groups.  See the remark following the proof of Corollary
\ref{cor:basepointInvariance}.  Bloom offered another proof for odd Khovanov homology using techniques
which have been since unexplored, as best we can tell.  It is interesting to
consider their relevance their relevance to other Khovanov-Floer theories.

\subsection{The structure of \Szabo homology}\label{subsec:introSzabo}

\Szabo homology is a link homology theory which interpolates between the
combinatorial and analytic Khovanov-Floer theories.  To a link diagram
$\diagram$ it assigns a filtered chain complex $\CSz(\diagram)$ whose underlying vector
space is identical to the Khovanov chain group.  Its homology, $\Sz(\diagram)$, is a link invariant.  The differential is
equal to the Khovanov differential plus maps along the diagonals of the cube of the
resolutions of $\diagram$.  There is a spectral sequence from $\Kh(\diagram)$ to
$\Sz(\diagram)$ which is formally similar to \Ozsvath-\Szabo's
spectral sequence from $\Kh(\diagram)$ to $\HF(\Sigma(-\diagram))$.  Conjecturally, this
similarity is more than formal.

\begin{Conj}[Szabo~\cite{Szabo2015}, Seed~\cite{Seed2011}]
  $\Sz(\diagram) \cong \HF(\Sigma(-\diagram))$, and the \Ozsvath-\Szabo spectral
  sequence agrees with the Leray spectral sequence from $\Kh(\diagram)$ to $\Sz(\diagram)$.
\end{Conj}

Seed used a computer program to provide numerical evidence for the conjecture.
Drawing on these computations and the spirit of the conjecture above, he made several
other conjectures about $\Sz$.

\begin{Conj}[\cite{Seed2011}]\leavevmode
  \begin{enumerate}
    \item Write $\rCSz(\diagram,p)$ for the \Szabo homology of
      $\diagram$ \emph{reduced at $p$}.  Then $\rCSz(\diagram,p) \cong \rCSz(\diagram,p')$ for any two points $p$, $p'$
      which are not double-points of $\diagram$.
    \item (``Twin arrows'') Let $E^k$ be the Leray spectral sequence from the homological
      filtration on $\CSz(\diagram)$.  Then
      \[
        E^k(\diagram) \cong \tilde{E}^k(L)\{-1\} \oplus \tilde{E}^k(L)\{1\}
      \]
      where $\tilde{E}^k$ denotes the spectral sequence on $\rCSz(\diagram)$ and
      $\{\pm 1\}$ denotes a shift in the quantum grading.
    \item Let $K$ be a knot.  Then $E^k(K)$ is invariant under mutation for $k
      \geq 2$.
    \item \Szabo homology is isomorphic to mirror \Szabo homology.
  \end{enumerate}
\end{Conj}

We prove the first and second of Seed's conjectures in the course of proving Theorem
\ref{thm:main}, and of course the third is the main subject of this
paper.\footnote{The first conjectured property is part of Lambert-Cole's definition
of an extended Khovanov-Floer theory and therefore plays a central rule in
proving the third conjecture.  But Seed's conjecture predates Lambert-Cole's
argument by six years -- he likely did not know that they were related in this
way!}  We do not have anything original to say about the fourth conjecture and
include it only for completeness.

\section*{Acknowledgements}

I am grateful to Peter Lambert-Cole for several helpful discussions.

\section{Strong Khovanov-Floer theories}\label{sec:skfts}

The following definition is from \cite{Saltz2017a}, following~\cite{BaldwinHeddenLobb2015}.

\begin{Def}\label{def:skft}
  A \emph{strong Khovanov-Floer theory} is a rule which assigns to a link
  diagram $\diagram$ and some auxiliary data $A$ a filtered chain complex
  $\skft(\diagram,A)$ so that
  \begin{enumerate}
    \item For any two collections of auxiliary data $A_\al$, $A_\be$ there is a
      filtered chain homotopy equivalence
      \[
        a_{\al}^{\be} \co \skft(\diagram,A_\al) \to \skft(\diagram, A_\be)
      \]
      so that the collection of all such complexes over all choices of auxiliary
      data and all the maps $a_{\al}^\be$ is a transitive system.  Write
      $\skft(\diagram)$ for the canonical representative (i.e. the inverse limit
      of the system).
    \item If $\diagram$ is a crossingless diagram of the unknot, then
      $H(\skft(\diagram)) \cong \Kh(\diagram)$.
    \item Let $\diagram \cup \diagram'$ be a disjoint union of diagram.  Then
      \[
        \skft(\diagram \cup \diagram') \simeq \skft(\diagram) \otimes \skft(\diagram').
      \]
    \item Suppose that $\diagram'$ is obtained from $\diagram$ by a
      diagrammatic handle attachment.  There is a function $\phi$ from the
      auxiliary data for $\diagram$ to the auxiliary data for $\diagram'$ and
      a map
      \[
        \handle[A_\al,\phi(A_\al),B] \co \skft(\diagram,A_\al) \to \skft(\diagram',\phi(A_{\al'}))
      \]
      where $B$ is some additional auxiliary data.  For fixed $B$, these maps
      form a map of transitive systems and therefore form a map
      $\handle[\beta] \co \skft(\diagram) \to \skft(\diagram')$.  For any two
      sets of additional data $B$ and $B'$, $\handle[\beta] \simeq
      \handle[\beta']$.
  \item Let $U$ be a crossingless diagram of an unknot.  Then $\skft(U)$
    is a Frobenius algebra with operations given by the handle attachment
    maps.  We say that $\skft$ is a strong Khovanov-Floer theory \emph{over} this Frobenius algebra. 
  \item If $\diagram'$ is obtained from $\diagram$ by a planar isotopy,
    then $\skft(\diagram)$ is filtered chain homotopy equivalent to
    $\skft(\diagram')$.
  \item Let $\diagram$ be the disjoint union of two diagrams $\diagram_0$
    and $\diagram_1$.  Let $\Sigma$  a diagrammatic cobordism from
    $\diagram$ to $\diagram'$.  Suppose that $\diagram'$ is the disjoint
    union of $\diagram'_0$ and $\diagram'_1$ and that $\Sigma$ is the
    disjoint union of $\Sigma_0$ and $\Sigma_1$, where $\Sigma_i$ is a
    cobordism from $\Sigma_i$ to $\Sigma'_i$.  Then $\skft(\Sigma) \simeq
    \skft(\Sigma_0) \otimes \skft(\Sigma_1)$.
  \item Handle attachment maps with disjoint supports commute up to
      filtered chain homotopy.  The map attached to a pair of canceling diagrammatic handle attachments is chain homotopic to the identity. 
\end{enumerate}
\end{Def}

Khovanov homology, Bar-Natan homology, Heegaard Floer homology of
branched double covers, singular instanton homology, and \Szabo homology are
strong Khovanov-Floer theories.  Every strong Khovanov-Floer theory we
are aware of satisfies an additional condition.
\begin{Def}\label{def:conic}
 Let $\diagram$ be a link diagram with crossings.  Pick a crossing and write
 $\diagram_0$ and $\diagram_1$ for the zero- and one-resolution of that
 crossing, respectively.  Let $\skft$ be a strong Khovanov-Floer theory.
 $\skft$ is \emph{conic} if
 \[
   \skft(\diagram) \simeq \cone(\handle[c] \co \skft(\diagram_0) \to \skft(\diagram_1))
 \]
 where $\handle[c]$ is the diagrammatic one-handle attachment along the arc
 shown in Figure \ref{fig:1handle}.
\end{Def}

\begin{figure}[h]
  \centering
  \includegraphics[width=.6\linewidth]{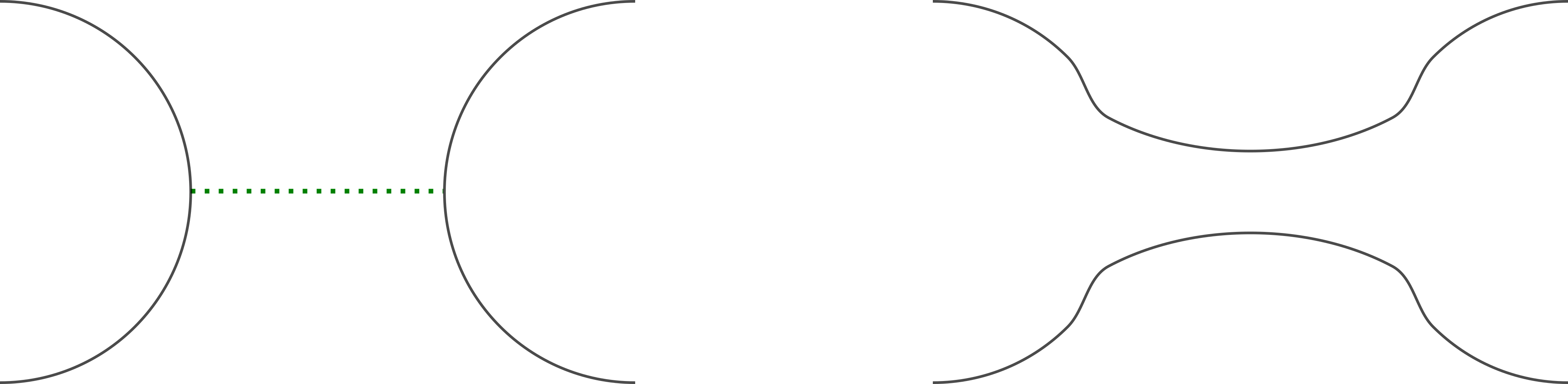}
  \caption{A diagrammatic one-handle attachment along the dotted green arc.}
  \label{fig:1handle}
\end{figure}

Being conic essentially means that $\skft(\diagram)$ can be computed from a cube
of resolutions. The central technical result of \cite{Saltz2017a} connects these
theories to Bar-Natan's \emph{formal Khovanov bracket} $\BN{-}$, see ~\cite{BarNatan2005}.  
\begin{Prop}\label{prop:mainTech}
  Let $\skft(-)$ be a conic, strong Khovanov-Floer theory.  Let $\BN{-}$ be
  Bar-Natan's formal Khovanov bracket.  $\skft(-)$ factors through $\BN{-}$ as a functor.

  More precisely, recall that $\BN{-}$ is a functor from a category of diagrams and cobordisms to the category $\matcob$.  There is a
  functor $F_\skft$ from a certain subcategory of $\matcob$ to $\Kom$ so that
  $\skft = F_\skft (\BN{-})$.
\end{Prop}

This is the point of Definition \ref{def:skft}: theorems about the formal
Khovanov bracket become theorems about every strong Khovanov-Floer theory.  So to prove the conjectures,
we will reframe the proofs of the equivalent facts for Khovanov homology in
terms of the formal bracket.

\section{The basepoint action and dotted cobordisms}\label{sec:basepointActions}

By a \emph{basepoint} of a link diagram $\diagram$ we mean a point which is not a double point.

\begin{Def}\label{def:basepointAction}
Let $\skft$ be a strong Khovanov-Floer theory over the Frobenius algebra $R = \bF[X]/(r(X))$ for some polynomial $r$. Let $U$ be a crossingless diagram of the unknot. Fix a chain homotopy equivalence $\skft(U) \simeq R$.  Multiplication by $X$ induces a chain homotopy class of map
\[
  \chi \co \skft(\circ) \to \skft(\circ).
\]
Let $\diagram$ be a link diagram and let $p \in \diagram$ be a basepoint. Define
\[
  \chi_p \co \skft(\diagram) \to \skft(\diagram)
\]
as the composition
\begin{align*}
  \skft(\diagram) \arrowunder{Z} \skft(\diagram) \otimes \skft(\circ) \arrowunder{\Id \otimes \chi} \skft(\diagram) \otimes \skft(\circ) \arrowunder{\handle[\gamma]} \skft(\diagram)
\end{align*}
The first map is a zero-handle attachment near $p$.  The third map is a one-handle attachment along an embedded arc $\gamma$ which connects the new component to $p$. 
\end{Def}

\begin{Lem}\label{lem:xAction}
  Let $p \in \diagram$ be a basepoint.  The map $X_p$ defines an action of $R$ on $\skft(\diagram)$. 
\end{Lem}
\begin{proof}   
  It suffices to show that $r(\chi_p) \simeq 0$.  Let $p \in U$.  By Condition 5 of
  Definition \ref{def:skft}, the map $r(\chi_p)$ induces $r(\chi)$ on $R$ via the
  chain homotopy equivalence $\skft(U) \simeq R$.  Therefore $r(\chi_p)$ is null-homotopic on $\skft(\circ)$.

  Observe that $\chi_p \circ \chi_p$ is chain homotopic to the following composition:
 \begin{align*}
   \skft(\diagram) &\arrowunder{Z} \skft(\diagram) \otimes \skft(U) \arrowunder{Z} \skft(\diagram) \otimes \skft(U) \otimes \skft(U) \\
                   &\arrowunder{\Id \otimes X \otimes X} \skft(\diagram) \otimes \skft(U) \otimes \skft(U)  \arrowunder{\Id \otimes \handle[\ga']} \skft(\diagram) \otimes \skft(U) \\
   &\arrowunder{\handle[\ga]} \skft(\diagram)
 \end{align*}
 Here $\ga'$ is an embedded arc connecting the two copies of $U$.  This follows
 from writing down the definition of $\chi_p \circ \chi_p$ and swapping the order of the two one-handles.  In general, $r(\chi_p)$ is chain homotopic to a map
 \[
   \skft(\diagram) \arrowunder{\Id \otimes Z} \skft(\diagram) \otimes \skft(\circ) \arrowunder{\Id \otimes r(\chi)} \skft(\diagram) \otimes \skft(\circ) \arrowunder{\handle[\gamma]} \skft(\diagram)
 \]
where $\gamma$ connects the new component to $p$.  So $r(X) \simeq 0$
implies that $r(X_p) \simeq 0$.
\end{proof}

This is not necessarily the Khovanov basepoint action, even if $\skft(\diagram)$
can be arranged to have the same rank as $\CKh(\diagram)$.  It is a fun exercise
to characterize the basepoint action on $\CSz(\diagram)$ in terms of \Szabo's
decorations and configurations.  (See the proof of Corollary
\ref{cor:basepointInvariance} for the answer.)

\subsection{Dotted cobordisms}\label{subsec:dottedCobordisms}

Over $\bF$, the map $\Id \otimes X$ in Definition \ref{def:basepointAction} does
not have an obvious cobordism-theoretic interpretation.\footnote{Over a ring
  without 2-torsion, multiplication by $X$ can be represented as attaching a
  tube.} In~\cite{BarNatan2005}, multiplication by $X$ is represented by dots on
a cobordism. In this section we check carefully that the dotted cobordism regime
makes sense for  strong Khovanov-Floer theories.

Using the results of \cite{Saltz2017a} it suffices to define dotted cobordisms embedded in $\br^2 \times I$.  A \emph{dotted zero-handle attachment} is a zero-handle attachment with a dot at the critical point.  To this we assign the map
\[
  \skft(\diagram) \arrowunder{Z} \skft(\diagram) \otimes \skft(\circ) \arrowunder{\Id \otimes X} \skft(\diagram) \otimes \skft(\circ) \longrightarrow \skft(\diagram \cup \circ) 
\]
A \emph{dotted two-handle attachment} is two-handle attachment with a dot at the
critical point.  To it we assign the dual of the zero-handle map.  Let $W \co
\diagram \to \diagram'$ be a cobordism without critical points.  Put a dot at
the point $p$ on $W$.  This is short-hand for a dotted zero-handle attachment near $p$ and a one-handle connecting the new component and $p$.  We call this a \emph{generic dot}.

\begin{Lem}\label{lem:dots}
  Let $W \co L \to L'$ be a link cobordism in $\br^2 \times I$ in Morse position.  Let $p \in W$ be an interior point which is not the critical point of a one-handle attachment.  Let $q$ be another such point on the same component of $W$.  Write $W_p$ and $W_q$ for the cobordisms with dots at $p$ and $q$.  Then $\skft(W_p) \simeq \skft(W_q)$.
\end{Lem}
\begin{proof}
  Connect the dots: as long as $p$ and $q$ lie on the same component of $W$, there is a path between them through generic points except possibly at $p$ and $q$.  It therefore suffices to show that $\skft(W_p) \simeq \skft(W_q)$ in the cases that
  \begin{itemize} 
  \item $p$ and $q$ are generic and there is a path between them which does not pass a critical level.
  \item $p$ and $q$ are generic and the path between them passes through a single critical level.
  \item $p$ is the critical point of a zero- or two-handle attachment, $q$ is generic, and there is a path between them which does not pass a critical level.
  \end{itemize}
  The first point is clear.  Suppose that $p$ and $q$ are separated by an index one critical point so that
  $p$ is above it.  Replace the generic dot at $p$ with a dotted zero-handle and
  a canceling one-handle.  Slide the one-handle past the index one critical point.  Now
  slide the dotted two-handle past the index one critical point.  These moves do
  not affect the chain homotopy type of $\skft(W_p)$ because of conditions
  7 and 8 of Definition \ref{def:skft}.  So by the first point, $\skft(W_p)
  \simeq \skft(W_q)$.  
  
  Suppose that $p$ lies at the critical point of a zero-handle attachment
  and that $q$ is a generic dot which lies just next to it.  Then $\skft(W_p)
  \simeq \skft(W_q)$ using Condition 8 of Definition \ref{def:skft}.  A similar argument applies to two-handle attachments. 
 \end{proof}

The arguments of \cite{Saltz2017a} imply the following Proposition.

\begin{Prop}\label{prop:dottedSFKT}
  Let $\cob_{\bullet}$ be the category whose objects are planar diagrams and
  whose morphisms are dotted and undotted cobordisms in $\br^2 \times I$.  There
  is a dotted Khovanov bracket $\BN{-}_{\bullet}$ defined in\cite{BarNatan2005}.  

  Every conic strong Khovanov-Floer theory factors through $\BN{-}_{\bullet}$.
\end{Prop}

The only thing to check is that dotted handles slide past each other.  This is an implication of condition 7 of \ref{def:skft}.  From Lemma \ref{lem:dots} and the 4Tu relation we get the \emph{dotted
  neck-cutting relation} of Figure \ref{fig:neckcutting} and~\cite{BarNatan2005}.
\begin{figure}[h]
  \centering
  \def\svgwidth{.5\linewidth}
  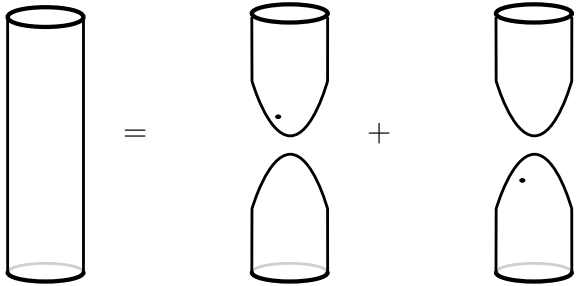
  \caption{The dotted neck-cutting relation.}
 \label{fig:neckcutting} 
\end{figure}

\subsection{Reduced theories}\label{subsec:reduced}

\begin{Def}\label{def:reduced}
  Let $\skft$ be a strong Khovanov-Floer theory over $\bF[X]/(X^2)$.  The complex $\Im(X_p)$ is called \emph{$\skft$ reduced at $p$} or \emph{the reduced version of $\skft$}.  We write often use the notation $\rskft_p$ or just $\rskft$.
\end{Def}

Now we can prove Seed's first two conjectures.

\begin{Prop}\label{prop:invariance}
  Let $\skft$ be a conic, strong Khovanov-Floer theory over $\bF[X]/(X^2)$.  Let $p$ and $q$ be two basepoints on $\diagram$.  Then $\rskft_p(\diagram) \simeq \rskft_q(\diagram)$ and
  \[
    \skft(\diagram) \cong \rskft(\diagram) \oplus \rskft(\diagram).
  \]
\end{Prop}

\begin{proof}
  First we define a map 
  \[
    N \co \BN{\diagram} \to \BN{\diagram}.
  \]
  If $\diagram$ is a crossingless unknot, then define $N$ to be the composition
  of a two-handle attachment with a zero-handle attachment, see Figure \ref{fig:N}.  Extend $N$ by the Liebniz rule to crossingless unlinks; for a crossingless diagram with $r$ components, $N$ is the sum of $r$ cobordisms which consist of $r-1$ cylinders and one cobordism like the one in Figure \ref{fig:N}.  Now suppose that $N$ is defined for links with up to $k$ crossings and let $\diagram$ be a link diagram with $k+1$ crossings.  Conicity implies that 
  \[
    \BN{\diagram} = \cone(\handle[k+1] \co \BN{\diagram_0} \to \BN{\diagram_1}).
  \]
  \begin{figure}[t]
    \centering
    \includegraphics{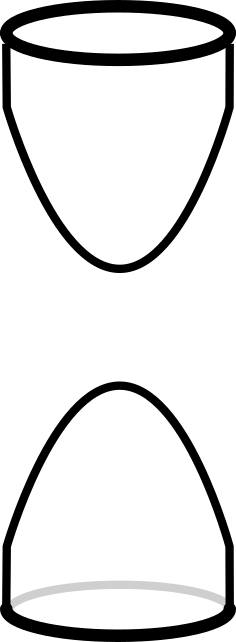} 
    \caption{The morphism $N$ applied to a crossingless unknot.}
    \label{fig:N}
  \end{figure} 
  $N$ is defined on $\BN{\diagram_0}$ and $\BN{\diagram_1}$ by hypothesis.  It defines a chain map on $\BN{\diagram}$ if it commutes with $\handle[k+1]$.  Figure \ref{fig:Nhandle} shows that it does.  The diagram describes four different cobordisms.  Label the circles 1 to 4 from left to right.  For $i, j$ distinct integers between $1$ and $4$, write $N_{ij}$ for the cobordism given by attaching a tube between circle $i$ and circle $j$.  The \emph{$4Tu$ relation} states that
  \[
    N_{12} + N_{23} + N_{34} + N_{41} = 0,
  \]
  taking the cobordisms to be morphisms in the linear cobordism category $\matcob$.  Observe that $N_{12} + N_{34}$ is a summand of $N \circ \handle[k+1]$, that $N_{41}$ is a summand of $\handle[k+1] \circ N$, and that $N_{23} = 0$.  The remaining summands cancel in pairs because the handle attachments belonging to $N$ have support disjoint from the support of $\handle[k+1]$.  Therefore $N$ and $\handle[k+1]$ commute.
  \begin{figure}[b]
    \centering
    \includegraphics[width=.5\linewidth]{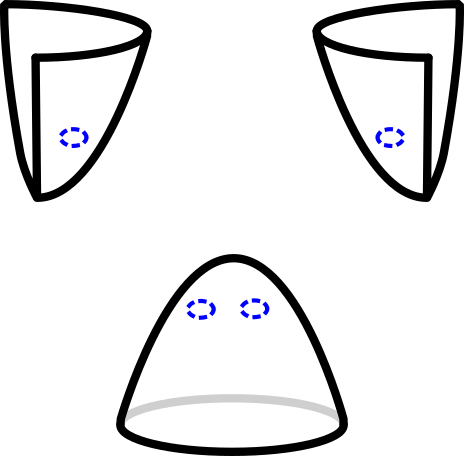}
    \caption{}
    \label{fig:Nhandle}
  \end{figure}

  Figure \ref{fig:bigPic} shows that
  \begin{align*}
    \chi_pN\chi_qN\chi_p &= \chi_p \\
    \chi_qN\chi_pN\chi_q &= \chi_q.
  \end{align*}
  \begin{figure}[h]
    \centering
    \def\svgwidth{.75\linewidth}
    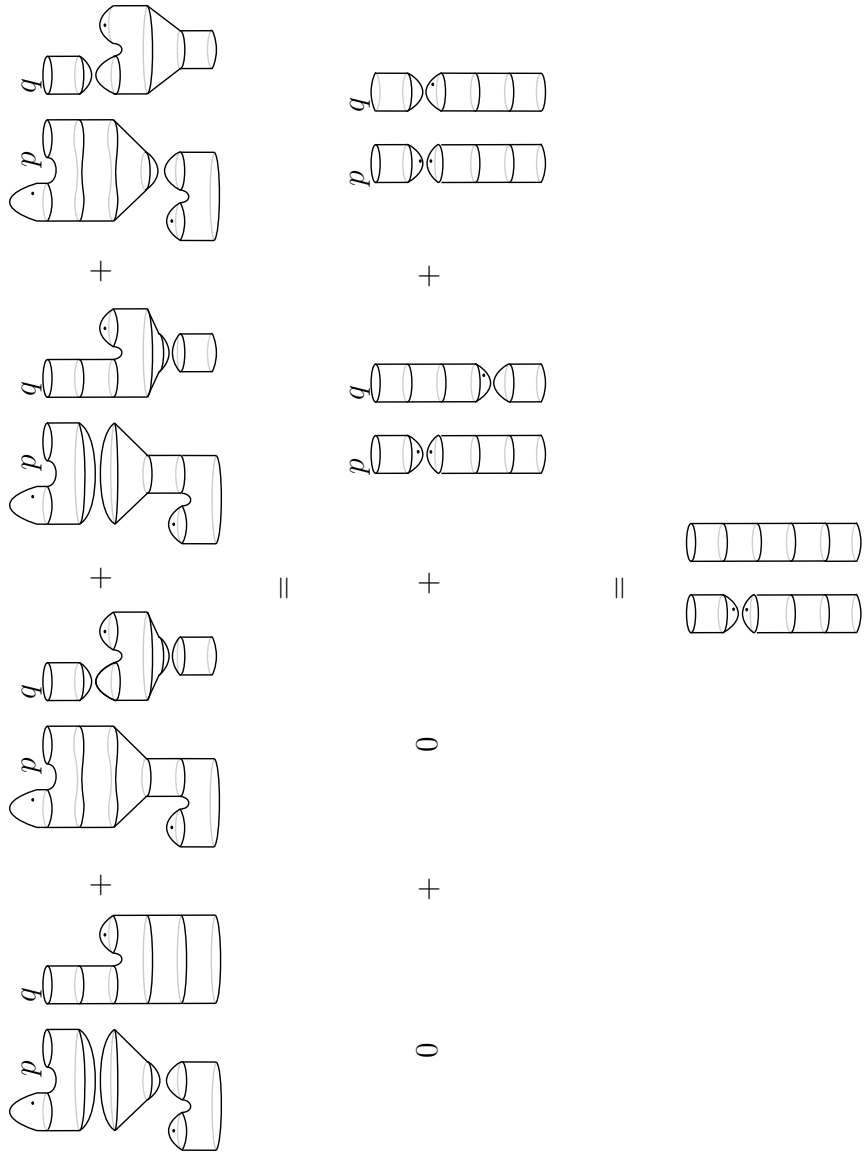  
    \caption{The calculation that $\chi_pN\chi_qN\chi_p = \chi_p$. It would be
      more precise to replace the circles labeled by $p$ and $q$ with arcs as in Figure \ref{fig:Nhandle} -- one-handle attachments between those circles may represent either merges or splits.}
    \label{fig:bigPic}
  \end{figure}
  
  Write $\nu = \skft(N)$ following~\cite{Shumakovitch2014} (or~\cite{Benheddi2017} for a very clear exposition).  It follows that
  \begin{align*}
    X_p \nu X_q \nu X_p &\simeq X_p \\
    X_q \nu X_p \nu X_q &\simeq X_q.
  \end{align*}
  $X_q \nu$ restricts to a map $\Im(X_p) \to \Im(X_q)$.  The equations above show that this restriction has chain homotopy inverse $X_p\nu$, and therefore $\rskft_p(\diagram) \simeq \rskft_q(\diagram)$.

  We have shown that the short exact sequence
  \[
    0 \longrightarrow \rskft_p(\diagram) \hookrightarrow \rskft(\diagram) \arrowunder{X_p} \rskft(\diagram) \to 0
  \]
  splits with section $\nu$.  The second statement follows.
\end{proof}

\begin{Cor}\label{cor:basepointInvariance}
  $\rSz(\diagram,p) \cong \rSz(\diagram,p')$ for any two basepoints $p$, $p'$ of
  $\diagram$.
\end{Cor}
\begin{proof}
  This follows immediately from the previous proposition as long as reduced
  \Szabo homology as defined by \Szabo agrees with the reduced strong
  Khovanov-Floer theory given by \Szabo homology.  \Szabo defines
  $\rCSz(\diagram,p)$ as the subcomplex of $\CSz(\diagram)$ in which every
  circle containing $p$ is labeled with $v_-$.  So we will show that $\rCSz(\diagram,p) = \Im(X_p)$.
 
  The action of $X_p$ on $\CSz(\diagram)$ can be written
  \[
    X_p = X_{p,0} + X_{p,1} + \cdots + X_{p,c}
  \]
  where $C$ is the number of crossings in $\diagram$.  Here $X_{p,k}$ is
  homological degree $k$ part of $X_p$; informally, it ``counts''
  the $k$-dimensional configurations counted in $\partial$ with the addition
  of the little circle near $p$ and the decoration $\gamma_p$.  All of these
  configurations must be of type E, and the central circle must be the one which
  contains $p$.  The \emph{filtration rule} of \cite{Szabo2015} implies that if $x \in
  \Im(X_p)$, then $x$ is a sum of canonical generators in which the circle with
  $p$ is labeled by $v_-$.  So $\Im(X_p) \subset \rCSz(\diagram,p)$.

  The argument that $\rCSz(\diagram,p) \subset \Im(X_p)$ is standard for bounded
  filtered complexes: the statement clearly
  holds for diagrams with no crossings.  Now suppose that it holds for diagrams
  with up to $c-1$ crossings, and let $\diagram$ be a diagram with $C$
  crossings.  Write $I_0$ for the all-zeroes resolution of $\diagram$.  Then the
  subcomplex of $\rCSz(\diagram,p)$ generated by all resolutions except $I_0$ is
  a subcomplex of $\Im(X_p)$.  Let $y \in \rCSz(\diagram,p)$ be a canonical
  generator and let $y^+$ be the canonical generator of $\CSz(\diagram,p)$
  obtained by replacing the $v_-$ at $p$ with a $v_+$.  Then
  \[
    X_p(y^+) = y + z
  \]
  where $z$ is a sum of canonical generators in resolutions other than $I_0$.
  By hypothesis, there is an element $z^+ \in \CSz(\diagram,p)$ so that
  \[
    X_p(y^+ + z^+) = y.
  \]
  Therefore $\rCSz(\diagram,p) \subset \Im(X_p)$.
\end{proof}

\begin{Rem}
  Let $p$ and $q$ be basepoints on opposite sides of a crossing $c$ of $\diagram$.
  In the formal Khovanov bracket it holds that $X_p \simeq X_q$.  The
  chain homotopy is given by the ``backwards map'' $\handle[c'] \co \BN{\diagram_1} \to
  \BN{\diagram_0}$, the one-handle attachment dual to $\handle[c]$.  Proving
  this statement for strong Khovanov-Floer theories would allow us to adapt
  Wehrli's proof of mutation-invariance.
\end{Rem}

Seed's second conjecture follows from the fact that the Leray spectral sequence
for a direct sum of complexes is isomorphic to the direct sum of the individual
spectral sequences and from the fact that $\nu$ has quantum grading $2$.

\section{Connected sums and mutation}\label{sec:connectedSum}

\begin{Prop}\label{prop:connectedSums}
  Let $\diagram$ and $\diagram'$ be link diagrams and write $\diagram \# \diagram'$ for their connected sum.  Let $\skft$ be a conic, strong Khovanov-Floer theory over $\bF[X]/(X^2)$.  Then
  \[
    \rskft(\diagram) \otimes_{\bF} \rskft(\diagram') \simeq \rskft(\diagram \# \diagram'). 
  \]
  Fix points $p \in \diagram$ and $q \in \diagram'$ so that there is an embedded arc $\gamma$ from $p$ to $q$.  Write $\nu$ and $\nu'$ for the $\nu$ maps on $\skft(\diagram)$ and $\skft(\diagram')$.  One isomorphism is
  \[
    \handle[\gamma] \circ (\Id \otimes \nu').
  \]
  It is chain homotopic to the isomorphism
  \[
    \handle[\gamma] \circ (\nu \otimes \Id).
  \]
\end{Prop}
\begin{proof}
  See Figures \ref{fig:connectedSum1}, \ref{fig:connectedSum2}, and
  \ref{fig:connectedSum3}.
  \begin{figure}[h]
    \Large
    \centering
    \def\svgwidth{.5\linewidth}
\begingroup%
  \makeatletter%
  \providecommand\color[2][]{%
    \errmessage{(Inkscape) Color is used for the text in Inkscape, but the package 'color.sty' is not loaded}%
    \renewcommand\color[2][]{}%
  }%
  \providecommand\transparent[1]{%
    \errmessage{(Inkscape) Transparency is used (non-zero) for the text in Inkscape, but the package 'transparent.sty' is not loaded}%
    \renewcommand\transparent[1]{}%
  }%
  \providecommand\rotatebox[2]{#2}%
  \newcommand*\fsize{\dimexpr\f@size pt\relax}%
  \newcommand*\lineheight[1]{\fontsize{\fsize}{#1\fsize}\selectfont}%
  \ifx\svgwidth\undefined%
    \setlength{\unitlength}{198.64279919bp}%
    \ifx\svgscale\undefined%
      \relax%
    \else%
      \setlength{\unitlength}{\unitlength * \real{\svgscale}}%
    \fi%
  \else%
    \setlength{\unitlength}{\svgwidth}%
  \fi%
  \global\let\svgwidth\undefined%
  \global\let\svgscale\undefined%
  \makeatother%
  \begin{picture}(1,1.10516265)%
    \lineheight{1}%
    \setlength\tabcolsep{0pt}%
    \put(0,0){\includegraphics[width=\unitlength,page=1]{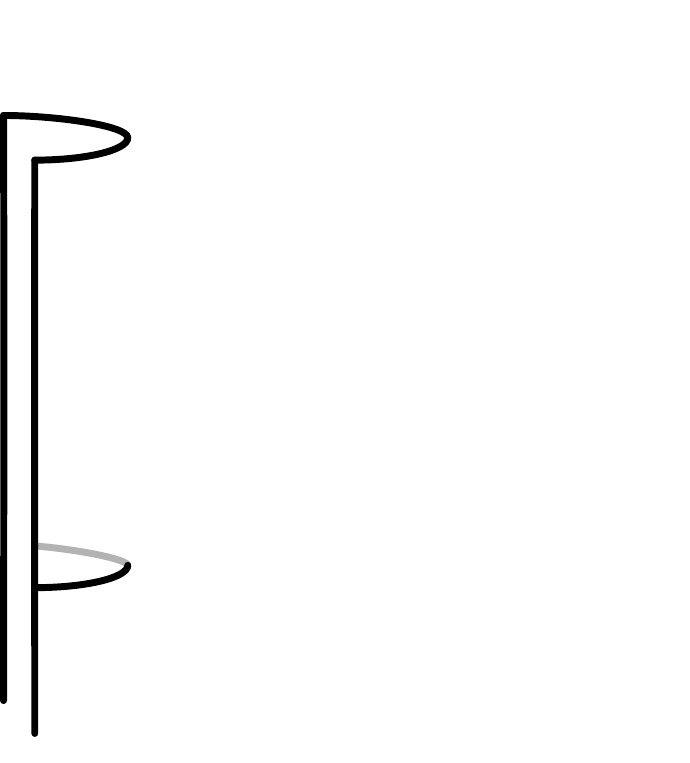}}%
    \put(0.05488965,0.97134936){\color[rgb]{0,0,0}\makebox(0,0)[lt]{\lineheight{1.25}\smash{\begin{tabular}[t]{l}$p$\end{tabular}}}}%
    \put(0.70376531,0.99041916){\color[rgb]{0,0,0}\makebox(0,0)[lt]{\lineheight{1.25}\smash{\begin{tabular}[t]{l}$q$\end{tabular}}}}%
    \put(0,0){\includegraphics[width=\unitlength,page=2]{connectedSum1.pdf}}%
  \end{picture}%
\endgroup%

    \caption{The interesting part of a cobordism which represents
      $\handle[\gamma] \circ (\Id \otimes \nu')$.  Of course $\nu$ is a sum of
      many cobordisms, but the rest are each equal to zero because they contain a
      surface with two dots.}
    \label{fig:connectedSum1}
  \end{figure}
  \begin{figure}[h]
    \Large
    \centering
    \def\svgwidth{.5\linewidth}
\begingroup%
  \makeatletter%
  \providecommand\color[2][]{%
    \errmessage{(Inkscape) Color is used for the text in Inkscape, but the package 'color.sty' is not loaded}%
    \renewcommand\color[2][]{}%
  }%
  \providecommand\transparent[1]{%
    \errmessage{(Inkscape) Transparency is used (non-zero) for the text in Inkscape, but the package 'transparent.sty' is not loaded}%
    \renewcommand\transparent[1]{}%
  }%
  \providecommand\rotatebox[2]{#2}%
  \newcommand*\fsize{\dimexpr\f@size pt\relax}%
  \newcommand*\lineheight[1]{\fontsize{\fsize}{#1\fsize}\selectfont}%
  \ifx\svgwidth\undefined%
    \setlength{\unitlength}{196.95474709bp}%
    \ifx\svgscale\undefined%
      \relax%
    \else%
      \setlength{\unitlength}{\unitlength * \real{\svgscale}}%
    \fi%
  \else%
    \setlength{\unitlength}{\svgwidth}%
  \fi%
  \global\let\svgwidth\undefined%
  \global\let\svgscale\undefined%
  \makeatother%
  \begin{picture}(1,0.51288233)%
    \lineheight{1}%
    \setlength\tabcolsep{0pt}%
    \put(0.07563981,0.3971554){\color[rgb]{0,0,0}\makebox(0,0)[lt]{\lineheight{1.25}\smash{\begin{tabular}[t]{l}$p$\end{tabular}}}}%
    \put(0.6896871,0.3971554){\color[rgb]{0,0,0}\makebox(0,0)[lt]{\lineheight{1.25}\smash{\begin{tabular}[t]{l}$q$\end{tabular}}}}%
    \put(0,0){\includegraphics[width=\unitlength,page=1]{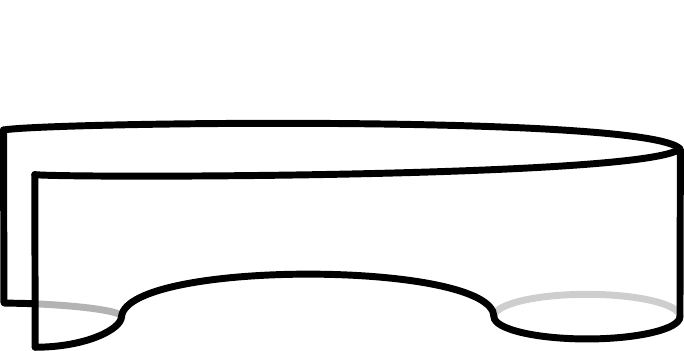}}%
  \end{picture}%
\endgroup%

    \caption{The inverse to $S$ -- the handle attachment dual to
      $\handle[\gamma]$.}
    \label{fig:connectedSum2}
  \end{figure}
  \begin{figure}[h]
    \Large
    \centering
    \def\svgwidth{.5\linewidth}
\begingroup%
  \makeatletter%
  \providecommand\color[2][]{%
    \errmessage{(Inkscape) Color is used for the text in Inkscape, but the package 'color.sty' is not loaded}%
    \renewcommand\color[2][]{}%
  }%
  \providecommand\transparent[1]{%
    \errmessage{(Inkscape) Transparency is used (non-zero) for the text in Inkscape, but the package 'transparent.sty' is not loaded}%
    \renewcommand\transparent[1]{}%
  }%
  \providecommand\rotatebox[2]{#2}%
  \newcommand*\fsize{\dimexpr\f@size pt\relax}%
  \newcommand*\lineheight[1]{\fontsize{\fsize}{#1\fsize}\selectfont}%
  \ifx\svgwidth\undefined%
    \setlength{\unitlength}{216.85701648bp}%
    \ifx\svgscale\undefined%
      \relax%
    \else%
      \setlength{\unitlength}{\unitlength * \real{\svgscale}}%
    \fi%
  \else%
    \setlength{\unitlength}{\svgwidth}%
  \fi%
  \global\let\svgwidth\undefined%
  \global\let\svgscale\undefined%
  \makeatother%
  \begin{picture}(1,1.13740575)%
    \lineheight{1}%
    \setlength\tabcolsep{0pt}%
    \put(0,0){\includegraphics[width=\unitlength,page=1]{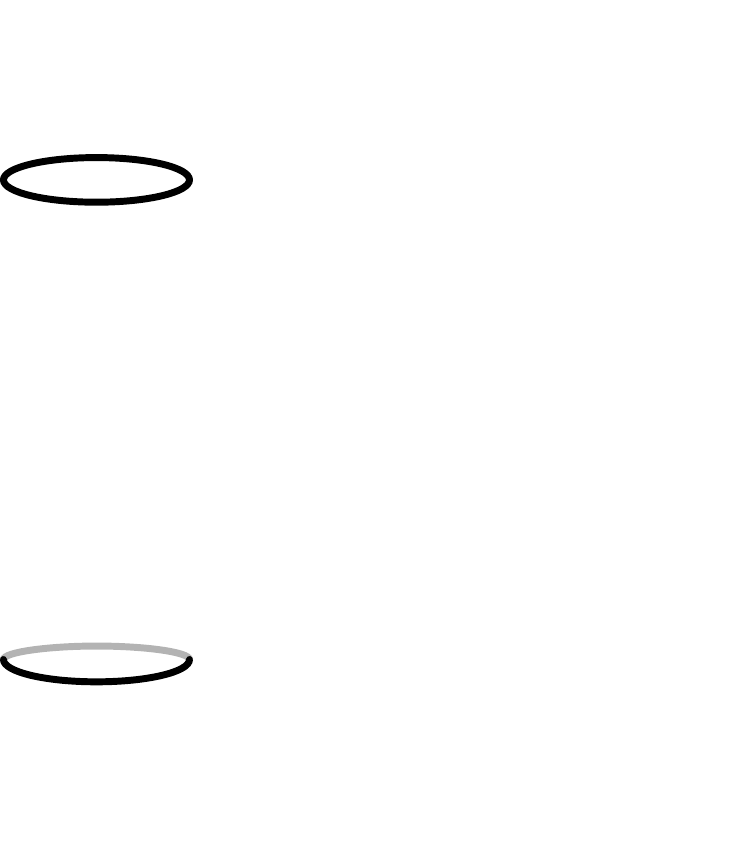}}%
    \put(0.00646283,1.03229977){\color[rgb]{0,0,0}\makebox(0,0)[lt]{\lineheight{1.25}\smash{\begin{tabular}[t]{l}$p$\end{tabular}}}}%
    \put(0.74757003,1.03229977){\color[rgb]{0,0,0}\makebox(0,0)[lt]{\lineheight{1.25}\smash{\begin{tabular}[t]{l}$q$\end{tabular}}}}%
    \put(0,0){\includegraphics[width=\unitlength,page=2]{connectedSum3.pdf}}%
  \end{picture}%
\endgroup%

    \caption{Apply the 4Tu relation to this picture to see that $\handle[\gamma]
      \circ (\Id \otimes \nu')$ is homotopic to $\handle[\gamma] \circ (\nu
      \otimes \Id)$.}
    \label{fig:connectedSum3}
  \end{figure}
\end{proof}

The following definition and result are due to Lambert-Cole in~\cite{2017arXiv170100880L}. 

\begin{Def}\label{def:ekft}
  An \emph{extended Khovanov-Floer theory} is a pair of a Khovanov-Floer theory and a reduced Khovanov-Floer theory $(\kft, \rkft)$ so that
  \begin{enumerate}
  \item $\kft(L) = \rkft(L \cup \circ, p)$.
  \item $\kft$ and $\rkft$ satisfy unoriented skein exact triangles.
  \item $\rkft(L,p) \cong \rkft(L,q)$.
  \end{enumerate}
\end{Def}

\begin{Thm*}[\cite{2017arXiv170100880L}]
  Extended Khovanov-Floer theories are invariant under Conway mutation.
\end{Thm*}

Now we can prove Theorem \ref{thm:main}. 

\begin{proof}[Proof of Theorem \ref{thm:main}]
In~\cite{Saltz2017a} we constructed a Khovanov-Floer theory for every strong Khovanov-Floer theory -- use the Leray spectral sequence given by the homological filtration on $\skft(\diagram)$.  For a strong Khovanov-Floer theory $\skft$ over $\bF[X]/(X^2)$, this Khovanov-Floer $\kft$ theory agrees with the one defined in~\cite{BaldwinHeddenLobb2015}.  

Every conic, strong Khovanov-Floer theory and its reduced version form an extended Khovanov-Floer theory under this construction.  The first requirement of Definition \ref{def:ekft} is straightforward.  The second is equivalent to conicity.  We have shown that $\rskft(\diagram,p) \simeq \rskft(\diagram,q)$.  It follows that $\rkft(\diagram,p) \cong \rkft(\diagram,q)$.
\end{proof}

 Our proof relies on Lambert-Cole's
theorem, but in fact Lambert-Cole's proof can be adapted to any strong
Khovanov-Floer theory.  In other words, the proof need not use any spectral
sequence techniques.

\bibliographystyle{plain}
\bibliography{bibliography}
\end{document}